\documentclass[12pt,leqno]{amsart}
\usepackage[all]{xy}
\usepackage{amsfonts,amsmath,oldgerm,amssymb,amsthm}
\usepackage{amssymb,amscd, lastpage, array, fancyheadings} %,array}
\usepackage{graphicx}
\usepackage{setspace}
\doublespace

\UseComputerModernTips
\usepackage{times}%helvetic}%zapfchan}
\newcommand{\lf}{\lfloor}
\newcommand{\rf}{\rfloor}

\newtheorem{theorem}{Theorem}

\newtheorem{lemma}[subsection]{{\bf Lemma}}
\newtheorem{coro}[subsection]{{\bf Corollary}}
\newtheorem{conj}[subsection]{{\bf Conjecture}}

\newcommand{\al}{\alpha}
\newcommand{\be}{\beta}
\newcommand{\ga}{\gamma}
\newcommand{\la}{\lambda}
\newcommand{\La}{\Lambda}
\newcommand{\om}{\omega}
\newcommand{\del}{\delta}
\newcommand{\es}{\epsilon}

\newcommand{\Z}{\mbox{$\mathbb Z$}}

\newcommand{\R}{\mbox{$\mathbb R$}}     % For Real numbers
\begin{document}
\title{Grimm's Conjecture and Smooth Numbers}
\author{Shanta Laishram}
\address{Stat-Math Unit\\
Indian Statistical Institute, New Delhi, India}
\email{shanta@isid.ac.in}
\author{M. Ram Murty}
\address{Department of Mathematics\\
Queen's University, Kingston, On, Canada}
\email{murty@mast.queensu.edu}

\subjclass[2010]{11N36 (primary),  11N25 (secondary)}
\keywords{Grimm's conjecture, smooth numbers, Selberg's sieve.}
\thanks{Research of the second author partially supported by an NSERC Discovery grant.}

\maketitle

\pagenumbering{arabic}
\pagestyle{headings}
\begin{abstract}
Let $g(n)$ be the largest positive integer $k$ such that
there are distinct primes $p_i$ for $1\leq i\leq k$ so that
$p_i |n+i$.  This function is related to a celebrated conjecture
of C.A. Grimm.  We establish upper and lower bounds for $g(n)$ by
relating its study to the distribution of smooth numbers. Standard conjectures
concerning smooth numbers in short intervals imply $g(n) =O(n^\epsilon)$
for any $\epsilon >0$. We also prove unconditionally that $g(n) =O(n^\al)$
with $0.45<\al <0.46$. The study of $g(n)$ and cognate
functions has some interesting implications for gaps between consecutive primes.
\end{abstract}

\section{Introduction}
In 1969, C.A. Grimm \cite{grim} proposed a seemingly innocent conjecture
regarding prime factors of consecutive composite numbers.  We begin by
stating this conjecture.
\par
Let $n\ge 1$ and $k\ge 1$ be integers.
\emph{Suppose $n+1, \cdots , n+k$ are all composite numbers.  Then there
are distinct primes $P_i$ such that $P_i|(n+i)$ for $1\le i\le k$.}
That this is a difficult conjecture having several interesting consequences
was first pointed out by Erd\"os and Selfridge \cite{erd}.
For example, the conjecture implies there is a
prime between two consecutive square numbers,  something
which is out of bounds for even the Riemann hypothesis.
In this paper, we will pursue this theme.  We will relate several
results and conjectures regarding smooth numbers (defined below) to Grimm's conjecture.
\par
To begin,
we say that Grimm's conjecture
holds for $n$ and $k$
if there are distinct primes $P_i$ such that
$P_i|(n+i)$ for $1\le i\le k$ whenever $n+1, \ldots , n+k$ are all
composites.
For positive integers $n>1$ and $k$, we say that \emph{$(n, k)$ has a prime representation
if there are distinct primes $P_1, P_2, \ldots , P_k$ with $P_j|(n+j)$, $1\le j\le k$.}
We define $g(n)$ to be the maximum positive integer $k$ such that $(n, k)$ has a
prime representation. It is an interesting problem to find the best possible
upper bounds and lower bounds for $g(n)$.  If $n'$ is the smallest prime
greater than $n$, Grimm's conjecture would imply that $g(n)> n' -n$.
On the other hand, it is clear that $g(2^m) < 2^m$ for $m >3$.
\par
The question of obtaining lower bounds for $g(n)$ was attacked using methods
from transcendental number theory by
Ramachandra, Shorey and Tijdeman \cite{rst75} who derived
$$g(n)\ge c\left(\frac{\log n}{\log \log n}\right)^3$$ for $n>3$ and an absolute
constant $c>0$.
In other words, for any sufficiently large natural number $n$,
$(n,k)$ has a prime representation if $k \ll (\log n /\log \log n)^3$.

We prove:

\begin{theorem}\label{gnbd}\hfill

\begin{itemize}
\item[$(i)$] There exists an $\al <\frac{1}{2}$ such that $g(n)<n^\al$ for sufficiently large $n$.
\item[$(ii)$] For $\es>0$, we have $|\{n\leq X: g(n)\ge n^\es\}|\ll X{\rm exp}(-(\log X)^{\frac{1}{3}-\es})$
where the implied constant depends only on $\es$.
\end{itemize}
\end{theorem}

We show in Section 3 that $0.45<\al<0.46$ is permissible in Theorem \ref{gnbd}(i).

For real $x, y$, let $\Psi(x, y)$ denote the number of positive integers $\le x$ all
of whose prime factors do not exceed $y$. These are $y$-\emph{smooth numbers} and have been
well-studied.
In 1930, Dickman \cite{dickman} proved that for any $\alpha \leq 1$,
$$\lim_{x \to \infty} \frac{\Psi(x,x^\alpha)}{x} $$
exists and equals $\rho(1/\alpha)$ where $\rho(t)$ is
defined for $t \geq 0$ as the continuous solution of the
equations $\rho(t) = 1$ for $0\leq t\leq 1$
and $-t\rho'(t) = \rho(t-1)$ for $t \geq 1$.  Later authors
derived refined results.
We refer to \cite{tenen} for an excellent survey on smooth numbers.
An important conjecture on smooth numbers in short intervals is the following.

\begin{conj}\label{conj1}
Let $\es >0$. For sufficiently large $x$, we have
\begin{align*}
\Psi(x+x^\es, x^\es) -\Psi(x, x^\es)\gg x^\es.
\end{align*}
\end{conj}

This is still open. Assuming Conjecture \ref{conj1}, we have the following.

\begin{theorem}\label{grim}
Let $\epsilon >0$. Then $g(n)<n^{\epsilon}$ for large $n$ assuming
Conjecture \ref{conj1}.
\end{theorem}

Let $p_i$ denote the $i$th prime. As a consequence of Theorem \ref{grim}, we obtain

\begin{coro}
Assume Grimm's conjecture and Conjecture \ref{conj1}. Then
for any $\epsilon >0$,
\begin{align}
p_{i+1}-p_i<p_i^{\es}
\end{align}
for sufficiently large $i$.
\end{coro}

If we assume Grimm's conjecture alone, then Erd\H{o}s and Selfridge\cite{erd}
have shown that
$$p_{i+1} - p_i \ll (p_i/\log p_i)^{1/2}, $$
which is something well beyond what the Riemann hypothesis would imply
about gaps between consecutive primes.
Indeed, the Riemann hypothesis implies an upper bound of $O(p_i^{1/2} (\log p_i))$.
It was conjectured by Cram\'er \cite{cramer} in 1936 that
$$p_{i+1} - p_{i} \ll (\log p_i)^2 $$
If Cram\'er's conjecture is true, then the result of Ramachandra, Shorey and
Tijdeman \cite{rst75} would imply Grimm's conjecture, at least for sufficiently
large numbers.
In \cite{laishram-shorey}, Laishram and Shorey verified Grimm's conjecture for
all
$n < 1.9 \times 10^{10}$.  They also checked that $p_{i+1} - p_i < 1 + (\log p_i)^2$ for
$i \leq 8.5 \times 10^8$.

It is worth mentioning that
there are several weaker versions of Grimm's conjecture that have also
been attacked using methods of transcendental number theory.
For an integer $\nu >1$, we denote
by $\omega(\nu)$ the number of distinct prime divisors of
$\nu$ and let
$\omega(1)=0$.
A weaker version of Grimm's conjecture states that \emph{if $n+1, n+2, \ldots, n+k$
are all composite numbers, then $\om(\prod^k_{i=1}(n+i))\ge k$}. This conjecture is also open
though much progress has been made towards it by Ramachandra, Shorey and Tijdeman \cite{rashti}.
\par
We define $g_1(n)$ to be the maximum positive integer $k$ such that
$$\om(\prod^l_{i=1}(n+i))\ge l$$ for all $1\le l\le k$. Observe that $g_1(n)\ge g(n)$. We prove

\begin{theorem}\label{g1bd}
There exists a $\ga$ with $0<\ga<\frac{1}{2}$ such that
\begin{align}
g(n)\le g_1(n)<n^{\ga}
\end{align}
for large values of $n$.
\end{theorem}

We show in Section 5 that $\ga=\frac{1}{2}-\frac{1}{390}$ is permissible.
This result will be proved as a consequence of the following theorem
which is of independent interest.

\begin{theorem}\label{g1bdlem}
Suppose there exists $0<\al<\frac{1}{2}$ and $\del >0$ such that
\begin{align}\label{pixal}
\sum_{j\le m^\al}\left\{\pi(\frac{m+m^\al}{j})-\pi(\frac{m}{j})\right\}\ge \del m^\al
\end{align}
holds for large $m$. Then $g_1(n)<n^{\gamma}$ with
\begin{align*}
\gamma =\max(\al, \frac{1-\del(1-\al)}{2-\delta})<\frac{1}{2}.
\end{align*}
for large $n$.
\end{theorem}

A conjecture coming from primes in short intervals states that(see for example Maier \cite{maier}):
\begin{align*}
\pi(x+x^\al)-\pi(x)\sim \frac{x^\al}{\log x} \ {\rm as} \ x\rightarrow \infty.
\end{align*}
Assuming this conjecture, we obtain for $m\rightarrow \infty$,
\begin{align*}
&\sum_{j\le m^\al}\left\{\pi(\frac{m+m^\al}{j})-\pi(\frac{m}{j})\right\}\sim
\sum_{j\le m^\al}\frac{\frac{m^\al}{j}}{\log \frac{m}{j}}=
\frac{m^{\al}}{\log m}\sum_{j\le m^\al}\frac{1}{j(1-\frac{\log j}{\log m})}\\
&\sim \frac{m^{\al}}{\log m} \int^{m^\al}_1\frac{dt}{t(1-\frac{\log t}{\log m})}.
\end{align*}
Taking $u=\frac{\log t}{\log m}$, we get
\begin{align*}
&\sum_{j\le m^\al}\left\{\pi(\frac{m+m^\al}{j})-\pi(\frac{m}{j})\right\}\\
&\sim m^{\al}\int^{\al}_0\frac{du}{1-u}=m^\al[-\log (1-u)]^{\al}_{0}=-m^\al\log (1-\al)
\end{align*}
as $m\rightarrow \infty$. Continuing as in the proof of Theorem \ref{g1bdlem},
we obtain $g_1(n)<n^{\al_1}$ with
\begin{align*}
\al_1 =\max(\al, \frac{1+(1-\al)\log (1-\al)}{2+\log (1-\al)}).
\end{align*}
Since $\log(1-\al)\approx -\al$ for $0<\al <1$, we see that
\begin{align*}
&\frac{1+(1-\al)\log (1-\al)}{2+\log (1-\al)}\approx \frac{1-\al(1-\al)}{2-\al}=
\frac{1}{2}(1-\al+\al^2)(1-\frac{\al}{2})^{-1}\\
&\approx \frac{1}{2}(1-\al+\al^2)(1+\frac{\al}{2})=\frac{1}{4}(2-\al+\al^2+\al^3)
\end{align*}
and the function $\frac{1}{4}(2-\al+\al^2+\al^3)$ attains its maximum at $\al=\frac{1}{3}$
where the value of $\al_1\approx 0.4567$. Hence, it is unlikely that we can get a result with
$g_1(n)<n^{\ga}$ with $\ga<.4567$, by these methods. As such, this value $g_1(n)=O(n^\al)$
seems to agree with the permissible value of $0.45<\al<0.46$ in $g(n)=O(n^{\al})$.

It was noted by Erd\"os and Selfridge
in \cite{erd} that ``\emph{the assertion
$\ga <\frac{1}{2}$ seems to follow from a recent result of Ramachandra
\cite{rama} but we do not give the details here}.'' In \cite{erd-pom},
Erd\"os and Pomerance
noted again that ``\emph{Indeed from the proof in \cite{rama}, it
follows that there is an $\al > 0$ such that for all large $n$ a positive
proportion of the integers in $(n,n+n^{\al}]$ are divisible by a prime
which exceeds $n^{\frac{15}{26}}$. Using this result with the method in
\cite{erd} gives $g(n) < n^{\frac{1}{2}-c}$ for some fixed $c > 0$
and all large $n$}.'' However there is no proof anywhere in the literature
about this fact. We give a complete proof in this paper by generalizing the result of
Ramachandra \cite{rama} in Lemma \ref{delx}.

\section{Preliminaries and Lemmas}

We introduce some notation.
We shall always write $p$ for a prime number. Let $\La(n)$ be the von Mangoldt
function which is defined as $\La(n)=\log p$ if $n=p^r$ for some positive integer
$r$ and $0$ otherwise. We write $\theta(x)=\sum_{p\le x}\log p$.
For real $x, y$, let $\Psi(x, y)$ denote the number of positive integers $\le x$ all
of whose prime factors do not exceed $y$.
We also write $\log_2 x$ for $\log \log x$.
We begin with some results from prime number theory.

\begin{lemma}\label{pix} Let $k, t\in \Z$ and $x\in \R$. We have
\begin{itemize}
\item[(i)] $\pi(x)<\frac{x}{\log x}(1+\frac{1.2762}{\log x})$ for $x>1$.
\item[(ii)] $p_t>t(\log t+\log_2 t-c_1)$ for some $c_1>0$ and for large $t$.
%\item[(ii)] $\theta (x)=x+O(x{\rm exp}(-c\sqrt{\log x})$ for some $c>0$.
\item[(iii)] $\theta (x)\le 1.00008x$ for $x>0$.
\item[(iv)] $\theta(p_t)>t(\log t+\log_2 t -c_2)$ for some $c_2>0$ and for large $t$.
\item[(v)] $k!>\sqrt{2\pi k}~e^{-k}k^{k}e^{\frac{1}{12k+1}}$ for $k>1$.
\end{itemize}
\end{lemma}

The estimate $(ii)$ is due to Rosser and Schoenfeld \cite{rosc}.
Inequalities $(i), (iii)$ and $(iv)$
are due to Dusart \cite{Dus}. The estimate $(v)$ is Stirling's formula, see \cite{rob}.

The following results are due to Friedlander and Lagarias \cite{frla}.

\begin{lemma}\label{smooth} Let $0<\es <1$ be fixed. Then there are positive
constants $c_0$ and $c_1$ depending only on $\es$ such that there are at most
$c_1X{\rm exp}(-(\log X)^{\frac{1}{3}-\es})$ many $n$ with $1\le n\le X$ which do not satisfy
\begin{align}\label{psies>}
\Psi(n+n^\es, n^\es)-\Psi(n, n^\es)\ge c_0n^\es.
\end{align}
\end{lemma}

\begin{lemma}\label{eta} There exist positive absolute constants $\al$ and $c_1$ with
$\frac{3}{8}<\al <\frac{1}{2}$ such that
\begin{align}\label{psies}
\Psi(n+n^\al, n^\al)-\Psi(n, n^\al)>c_1n^\al.
\end{align}
for sufficiently large $n$.
\end{lemma}

Lemma \ref{smooth} is obtained by taking $\al=\be =\es$ in \cite[Theorem 5]{frla} and Lemma
\ref{eta} is obtained by taking $x=n, y=z=n^\al $ with $\al=\frac{1}{2}-\frac{\eta}{2}$
in \cite[Theorem 2.4]{frla}. From \cite[Theorem 2]{harman} and the remarks after that, a
permissible value of $\al$ in Lemma \ref{eta} is given by an $\al$ with $0.45<\al<0.46$.

The following is the key lemma which follows from the definition of $g(n)$
and relates the study of $g(n)$ to smooth numbers.

\begin{lemma}\label{key} Let $x, y, z\in \R$ be such that $\Psi(x+z, y)-\Psi(x, y)>\pi(y)$. Then
$g(\lf x\rf )<z$.
\end{lemma}

\begin{proof}  Let $x \leq n_1 < n_2 < \cdots < n_t \leq x+z $
be all $y$-smooth numbers with $t > \pi(y)$.  Then, $(n_1, n_t - n_1)$
does not have a prime representation.
In particular, $(\lf x \rf , \lf z \rf)$
has no prime representation.
Thus $g(\lf x \rf) < z$.
\end{proof}

The next result is a generalization of a result of Ramachandra \cite{rama}.

\begin{lemma}\label{delx}
Let $\frac{1}{33}<\la<\frac{1}{29}$. For $\al=\frac{1-\la}{2}$ and for sufficiently large $x$, we have
\begin{align}\label{r2.5}
\sum_{n\le x^{\al}}\left\{\pi(\frac{x+x^\al}{n})-\pi(\frac{x}{n})\right\}\ge (\frac{1}{4}+\frac{\la}{2}-\es') x^\al
\end{align}
where $\es'>0$ is arbitrary small.
\end{lemma}
We postpone the proof of Lemma \ref{delx} to Section 4.

\section{Proof of Theorems \ref{gnbd} and \ref{grim}}

\noindent
{\bf Proof of Theorem \ref{gnbd}:}
$(i)$ Let $\al$ be given by Lemma \ref{eta}. We apply Lemma \ref{key} by taking $x=n, z=y=n^\al$. Since
$\pi(y)=\pi(n^\al)<2\frac{n^\al}{\al \log n}<c_1n^\al$ for sufficiently large $n$, the assertion follows
from Lemma \ref{key} and Lemma \ref{eta}. As remarked after Lemma \ref{eta}, a permissible
value of $\al$ is given by $0.45<\al<0.46$.

\noindent
$(ii)$ Let $\es >0$ be given. By $(i)$, we may assume that $\es <\frac{1}{2}$. Since
$\pi(n^\es)<2\frac{n^\es}{\es \log n}<c_0 n^\es$ for sufficiently large $n$
where $c_0$ is given by Lemma \ref{smooth}, the assertion now follows from
Lemma \ref{key} by taking $x=n, z=y=n^\es$ and Lemma \ref{smooth}.
$\hfill \Box$

\vspace{1cm}

\noindent
{\bf Proof of Theorem \ref{grim}:}
Let $\es>0$ be given. We apply Lemma \ref{key} by taking $x=n, z=y=n^\es$. Since
$\pi(y)=\pi(n^\es)<2\frac{n^\es}{\es \log n}\ll n^\es$ for sufficiently large $n$, the assertion follows
from Lemma \ref{key} and Conjecture \ref{conj1}. $\hfill \Box$

\section{Proof of Lemma \ref{delx}}

We follow the proof of Ramachandra in \cite{rama} and fill in the
details as we go along. Let $\al <\frac{1}{2}$ and $0< \be <\frac{1}{2}$.
By taking
$\epsilon=x^{\al - 1}$ in \cite[Lemma 1]{rama}, we obtain
\begin{align}\label{rl1}
\sum_{n\le x^{1-\al}}\left\{\pi(\frac{x+x^\al}{n})-\pi(\frac{x}{n})\right\}\log \frac{x}{n}=(1-\al)x^\al \log x
+O(x^{\al}).
\end{align}
We divide the interval $[\be , 1-\al]$ as $0<\be=\be_0<\be_1<\ldots <\be_m=1-\al$ for some $m$. For $0< r < s < 1$, let
\begin{align}
S(r, s)=\sum_{x^r \le n\le x^s}\left\{\pi(\frac{x+x^\al}{n})-\pi(\frac{x}{n})\right\}\log \frac{x}{n}.
\end{align}
We would like to get an upper bound for $S(\be, 1-\al)=\sum^{m-1}_{i=0}S(\be_i, \be_{i+1})$.
We first prove the following lemma which is minor refinement of \cite[Lemma 3]{rama}.

\begin{lemma}\label{rlem3}
Let $x\ge 1$ and $1\le R\le S\le x^{1-\al}$. For an integer $d\ge 1$, let
\begin{align}
R_d=\sum_{R \le n\le S}\left\{\big{[}\frac{x+x^\al}{nd}\big{]}-\big{[}\frac{x}{nd}\big{]}\right\}.
\end{align}
Then
\begin{align}\label{rl3}
\begin{split}
\sum_{R \le n\le S}\left\{\pi(\frac{x+x^\al}{n})-\pi(\frac{x}{n})\right\}
&\le \frac{(2-\es ) x^\al}{\log z}\log (\frac{S}{R}+2)\left(1+O(\frac{1}{R}+\frac{1}{\log z})\right)\\
&+O(z\max_{d\le z} |R_d|)
\end{split}
\end{align}
where $z\ge 3$ is an arbitrary real number and $\es >0$ is arbitrary small.
\end{lemma}

\begin{proof}
Let $$T=\underset{R\le n\le S}{\cup}\left((\frac{x}{n}, \frac{x+x^\al}{n}]\cap \Z\right).$$
 From $T$, we
remove those which are divisible by primes $\le \sqrt{z}$ and let $T_1$ be the remaining set. We note that
for each $d$, the number of integers in $T$ divisible by $d$ is
\begin{align*}
\frac{x^\al}{d}\sum_{R\le n\le S} \frac{1}{n}+R_d
\end{align*}
Using Selberg's sieve as in \cite{rama}, we obtain the assertion of lemma.
\end{proof}

Let $\phi(u)=u-[u]-\frac{1}{2}$. Then we can write
$$\big{[}\frac{x+x^\al}{nd}\big{]}-\big{[}\frac{x}{nd}\big{]}=
\frac{x^{\alpha}}{nd}
-\phi(\frac{x+x^\al}{nd})+\phi(\frac{x}{nd}).$$
The following result is a restatement of \cite[Lemma 2]{rama} which follows from
a result of van der Corput (see \cite{rama}).

\begin{lemma}\label{rlem2}
Let $u\ge 1, V, V_1$ be real numbers satisfying $3\le V<V_1\le 2V, V_1\ge V+1$ and $u\le \eta \le 2u$. Then
\begin{align}
\sum_{V \le n\le V_1} \phi (\frac{\eta}{n})=O(V^{\frac{1}{2}}\log V+V^{\frac{3}{2}}u^{-\frac{1}{2}}+u^{\frac{1}{3}}).
\end{align}

\end{lemma}

To get an upper bound for $S(\be_i, \be_{i+1})$, we take $R=x^{\be_i}, S=x^{\be_{i+1}}$ in Lemma \ref{rlem3}. Recall that
$\be_{i+1}\le 1-\al$. We subdivide $(R, S]$ into intervals of type $(V, 2V]$ and at most one interval of type
$(V, V_1]$ with $V_1\le 2V$. We apply Lemma \ref{rlem2} twice by taking
$\eta=\frac{x}{d}$ and $\eta=\frac{x+x^\al}{d}$ to get
\begin{align*}
R_d&=O\left(\left(x^{\frac{1}{2}\be_{i+1}}+x^{\frac{1}{2}(3\be_{i+1}-1)}d^{\frac{1}{2}}+(\frac{x}{d})^{\frac{1}{3}}\right)
(\log x)^2\right)\\
&=O\left(\left(x^{1-\frac{3}{2}\al}d^{\frac{1}{2}}+x^{\frac{1}{3}}\right)(\log x)^2\right),
\end{align*}
since $\be_{i+1}\leq 1-\al$
Let $3\al-\frac{4}{3}<\del<\frac{5\al-2}{3}$ and take $z=x^\del$. Then
\begin{align*}
z\max_{d\le z}|R_d|=O(x^{1-\frac{3}{2}\al+\frac{3}{2}\del}(\log x)^2)
\end{align*}
and $1-\frac{3}{2}\al+\frac{3}{2}\del<\al$. From \eqref{rl3}, we obtain
\begin{align*}
\sum_{x^{\be_i} \le n\le x^{\be_{i+1}}}\left\{\pi(\frac{x+x^\al}{n})-\pi(\frac{x}{n})\right\}
\le \frac{2x^\al}{\del}(\be_{i+1}-\be_i).
\end{align*}
Therefore an upper bound for
\begin{align*}
\sum_{x^\be \le n\le x^{1-\al}}\left\{\pi(\frac{x+x^\al}{n})-\pi(\frac{x}{n})\right\}\log \frac{x}{n}
\end{align*}
is
\begin{align*}
&\frac{2x^\al \log x}{\del}\times \\
&\{(\be_1-\be_0)(1-\be_0)+(\be_2-\be_1)(1-\be_1)+\cdots +(\be_m-\be_{m-1})(1-\be_{m-1})\}.
\end{align*}
We take $\be_i$'s to be equally spaced and take $m$ sufficiently large. Since
\begin{align*}
\frac{2x^\al \log x}{\del}\int ^{1-\al}_{\be} (1-t)dt = \frac{x^\al\log x}{\del}(1-\al ^2-\be(2-\be)),
\end{align*}
we obtain with \eqref{rl1} that
\begin{align}\label{deles}
\sum_{n\le x^{\be }}\left\{\pi(\frac{x+x^\al}{n})-\pi(\frac{x}{n})\right\} \ge
(1-\al-\es' -\frac{1-\al ^2-\be(2-\be)}{\del})x^{\al}.
\end{align}
where $1-\frac{3}{2}\al+\frac{3}{2}\del<\al$ and $\es'>0$ is arbitrary small.

Let $\frac{1}{33}<\la <\frac{1}{29}$ and we put $\al=\be=\frac{1-\la}{2}$ and $\del=4\la$. Then
$1-\frac{3}{2}\al+\frac{3}{2}\del<\al$ and hence we obtain \eqref{r2.5} from \eqref{deles}.
$\hfill \Box$

\section{Proof of Theorem \ref{g1bdlem} and Theorem \ref{g1bd}}

We begin with the proof of Theorem \ref{g1bdlem}.

\begin{proof}
Recall that $g_1(n)$ is the largest integer $k$ such that
$$\omega(\prod_{i=1}^l (n+i)) \geq l $$
for $ 1\leq l \leq k$.  Suppose that $g_1(n) > n^\gamma$.  Then
$g_1(n) > n^\alpha$.  Let $k = [n^\alpha]$.  Then
$$\omega(P) \geq k, \quad P = \prod_{i=1}^k (n+i). $$
By (\ref{pixal}),
$$\sum_{j\leq k} \pi\left( {n+k \over j}\right) - \pi\left({n \over j}\right)
 \geq \delta k. $$
Now the intervals $[n, n+k]$, $[n/2, (n+k)/2]$, ... are disjoint intervals.
In fact, if we write $I_j = [n/j, (n+k)/j] = [a_j, b_j]$(say),
then it is easily seen $b_1 > a_1 > b_2 > a_2 > b_3 > a_3 \cdots $
by virtue of the condition that $k< n^\alpha$ with $\alpha < 1/2$.
A prime $q_i$ (say) lying in the interval $I_j$ satisfies
 $n < jq_i < n+k$ and consequently is a prime dividing $P$.
Since these primes $q_i$ are all distinct, and all of these
primes are greater than $n/k \geq n^{1-\alpha}$, we deduce that
there are at least $\delta k$ distinct primes greater than $n^{1-\alpha}$
dividing $P$. Let $\del'\geq \del$ be such that $\del'k=\lceil \del k\rceil$.
Since $\om(P)\geq k$, there are at least $(1-\del')k$ other primes dividing
$P$ and $(1-\del')k\in \Z$. Also $k!|P$ since $P$ is a product of $k$ consecutive
numbers. All the prime factors of $k!$ are less than or equal to
$k<n^\alpha<n^{1-\alpha}$ since $\alpha < 1/2$. Hence we get
\begin{align*}
P \geq k! \left(\prod_{k<p<p_{(1-\del')k}}p\right) (n^{1-\alpha})^{\delta' k}.
\end{align*}
Now we apply the bounds provided by Lemma \ref{pix}.
By Lemma \ref{pix} $(iii)$ and $(iv)$, we obtain
\begin{align*}
\log \left(\prod_{k<p\le p_{(1-\del')k}}p\right)=&\theta(p_{(1-\del')k})-\theta(k)\\
\ge &(1-\del')k\log (1-\del')k +(1-\del')k\{\log_2(1-\del')k -c_2\}\\
& -1.00008k\\
>&(1-\del')k\log (1-\del')k +k(c_3\log_2c_3k-c_4)
\end{align*}
where $c_3, c_4$ are positive constants. This together with $k!>(\frac{k}{e})^k$ by
Lemma \ref{pix} $(v)$ and $P<(2n)^k$ imply
\begin{align*}
2n&>\frac{k}{e}(1-\del')^{1-\del'} k^{1-\del'}c_5(\log c_3k)^{c_3}n^{\del'(1-\al)}\\
&=\frac{1}{e}(1-\del')^{1-\del'} c_5(\log c_3k)^{c_3}(\frac{k}{n^\gamma})^{2-\del'}n^{\ga (2-\del')}n^{\del'(1-\al)}\\
&>2n^{\ga (2-\del')+\del'(1-\al)}\ge 2n^{\ga (2-\del)+\del(1-\al)}2n
\end{align*}
for large $n$ since $\del'>\del$ and $1-\al>\frac{1}{2}>\ga$. This is a contradiction. Thus
$g_1(n)<k\le n^\al\le n^\ga$.
\end{proof}

\iffalse
\begin{lemma}
Suppose there exists $0<\al<\frac{1}{2}$ and $\del >0$ such that
\begin{align}\label{pixal}
\sum_{j\le m^\al}\left\{\pi(\frac{m+m^\al}{j})-\pi(\frac{m}{j})\right\}\ge \del m^\al
\end{align}
holds for large $m$. Then $g_1(n)<n^{\gamma}$ with
\begin{align*}
\gamma =\max(\al, \frac{1-\del(\frac{1}{2}+\la)}{A+1-\delta})<\frac{1}{2}.
\end{align*}
for large $n$.
\end{lemma}

\begin{proof}
Suppose $g_1(n)\ge n^{\gamma}$. Let $k=\lf n^\ga \rf$. Then $\om(P)\ge k$ where
$P:=P(n, k)=\prod^k_{i=1}(n+i)$. By \eqref{pixal}, there are at least $\del k$ primes
$\ge n^{\frac{1}{2}+\la}$ dividing $P$. Observe that $k!|P$. Since $\om(P)\ge Ak$, we
obtain
\begin{align*}
P\ge k!\left(\prod_{k<p\le p_{(A-\del)k}}p\right)n^{\del k(\frac{1}{2}+\la)}
\end{align*}
By Lemma \ref{pix}, we obtain
\begin{align*}
\log \left(\prod_{k<p\le p_{(A-\del)k}}p\right)&=\theta(p_{(A-\del)k})-\theta(k)\\
&\ge (A-\del)k\log (A-\del)k +(A-\del)k\{\log_2(A-\del)k -c_2\}-1.00008k\\
&>(A-\del)k\log (A-\del)k +k(c_3\log_2c_3k-c_4)
\end{align*}
where $c_3, c_4$ are positive constants. This together with $k!>(\frac{k}{e})^k$ and $P<(2n)^k$ imply
\begin{align*}
2n&>\frac{k}{e}(A-\del)^{A-\del} k^{A-\del}c_5(\log c_3k)^{c_3}n^{\del(\frac{1}{2}+\la)}\\
&=\frac{1}{e}(A-\del)^{A-\del} c_5(\log c_3k)^{c_3}(\frac{k}{n})^{A+1-\del}n^{\ga (A+1-\del)}n^{\del(\frac{1}{2}+\la)}\\
&>2n^{\ga (A+1-\del)+\del(\frac{1}{2}+\la)}=2n
\end{align*}
for large $n$. This is a contradiction.
\end{proof}
\fi

\noindent
{\bf Proof of Theorem \ref{g1bd}:} From Lemma \ref{delx}, we obtain \eqref{pixal} with
$\al=\frac{1-\al}{2}$ and $\del=\frac{1}{4}+\frac{\la}{2}-es'$ for some
$\frac{1}{33}<\la<\frac{1}{29}$. Now the assertion follows from \eqref{g1bdlem}.
Taking $\la=\frac{1}{30}+2\es'$ for instance, we get
$\ga\le \frac{1}{2}-\frac{1}{390}$.
$\hfill \Box$

\vspace{.75cm}
\noindent
{\bf Remark:} It is possible to improve the result we have obtained. However the improvement is
not substantial. Indeed the result of van der Corput has been improved and using methods of
Harman and Baker \cite{HarBak}, it is possible to obtain a small refinement. The details are rather
technical and will be discussed in a future paper by the junior author.

\section*{Acknowledgments}
%\noindent {\sl Acknowledgments.}
We thank the referee for helpful remarks and careful reading of our paper. We would also
like to thank Sanoli Gun and Purusottam Rath for their careful reading and corrections
on an earlier version of this paper.

\end{document}